\documentclass[11pt]{amsart}
\usepackage{amssymb,amsmath,amsthm,amsfonts,amsopn,url,color,
mathtools,microtype,MnSymbol,dutchcal,
mathrsfs}
\usepackage[colorlinks=true, pagebackref=true]%
{hyperref}
\usepackage{scalerel}
\usepackage{stackengine,wasysym}
\usepackage[shortlabels]{enumitem}

\usepackage[normalem]{ulem}
\usepackage[all]{xy}
\input xy
\xyoption{all}
\usepackage{amscd}
\usepackage{soul}

\theoremstyle{plain}
\newtheorem{thm}{Theorem}[section]

\newtheorem{prop}[thm]{Proposition}
\newtheorem{lemma}[thm]{Lemma}
\newtheorem{cor}[thm]{Corollary}

\theoremstyle{definition}
\newtheorem{defn}[thm]{Definition}
\newtheorem*{defn*}{Definition}
\newtheorem*{question*}{Question}
\newtheorem{question}{Question}
\newtheorem{example}[thm]{Example}
\newtheorem*{example*}{Example}
\newtheorem{rem}[thm]{Remark}
\newtheorem*{rem*}{Remark}

\newcommand{\field}[1]{\mathbb{#1}}
\newcommand{\N}{\field{N}}
\newcommand{\Z}{\field{Z}}
\newcommand{\Q}{\field{Q}}
\newcommand{\R}{\field{R}}
\newcommand{\F}{\field{F}}
\newcommand{\C}{\field{C}}
\newcommand{\A}{\field{A}}

\newcommand{\PP}{\field{P}}
\newcommand{\ideal}[1]{\mathfrak{#1}}
\newcommand{\m}{\ideal{m}}

\newcommand{\func}[1]{\mathrm{#1} \,}

\newcommand{\Spec}{\func{Spec}}

\newcommand{\ra}{\rightarrow}

\newcommand{\be}{\begin{enumerate}}
\newcommand{\ee}{\end{enumerate}}

\newcommand{\li}
 {\leftfootline}

\newcommand{\onto}{\twoheadrightarrow}

\newcommand{\into}{\hookrightarrow}

\newcommand{\cN}{\mathcal{N}}

\newcommand{\cO}{\mathcal{O}}

\renewcommand{\phi}{\varphi}
\DeclareMathOperator{\Frac}{Frac}

\DeclareMathOperator{\chr}{char}
\DeclareMathOperator{\Max}{Max}

\let\int\relax
\DeclareMathOperator{\int}{i}

\newcommand{\ua}{unit-additive}
\newcommand{\uad}{unit dimension}
\newcommand{\ut}
{UU}
\newcommand{\uty}
{the UU property}
\DeclareMathOperator{\Jac}{J}
\newcommand{\uasym}{{\rm{ua}}}
\newcommand{\alg}{{\rm{alg}}}

\DeclareMathOperator{\udim}{udim}
\newcommand{\puncaone}{{\mathbb A}^1_k \setminus \{\mathbf 0\}}

\author{Neil Epstein}
\address{Department of Mathematical Sciences \\ George Mason University \\ Fairfax, VA  22030}
\email{nepstei2@gmu.edu}

\author{Jay Shapiro}
\address{Department of Mathematical Sciences \\ George Mason University \\ Fairfax, VA  22030}
\email{jshapiro@gmu.edu}

\title[Rings where a non-nilpotent sum of units is a unit]{Rings where a non-nilpotent \\ sum of units is a unit}

\date{April 18, 2025}
\begin{document}

\begin{abstract}
A ring is \emph{\ua} if a sum of units is always either a unit or nilpotent. For example, $k[X]$ and $k[X]/(X^2)$ are unit-additive, but $\mathbb Z$ is not.  We prove a wide-ranging theorem about unit-additivity in semigroup rings, showing among other things that an affine semigroup ring $A[M]$ is unit-additive if and only if $A$ is unit-additive and $M$ has no nontrivial invertible elements.   Passing to algebraic geometry, we show that an irreducible affine variety $V$ over an algebraically closed field $k$ has unit-additive coordinate ring if and only if any polynomial mapping $V \ra k$ has a root.  This then places $\A^1_k$ into the class of varieties that satisfy a version of the Fundamental Theorem of Algebra.
Specializing to elliptic curves, we show that the affine coordinate ring of an elliptic curve is always unit-additive.  The concept of unit additivity leads to the related concept of \uad\ -- i.e. how far is an integral domain from being unit-additive? It turns out that rings of \uad\ 1 are of some interest, as they include the rings of integers of number fields, all power series rings, and most local rings.  We construct rings of all \uad s and show that in the affine setting, \uad\ is bounded above by Krull dimension. We also construct the \emph{unit-additive closure} of an integral domain $D$, being the smallest subring of the fraction field of $D$ that is unit-additive, as a localization at a certain multiplicative set in $D$. Throughout, we make connections with well-studied structures like PIDs, Euclidean domains, and the UU property.
\end{abstract}
\maketitle

\section{Introduction}
It is never true in a nonzero ring that every sum of units is a unit.  This is because if $u$ is a unit, then so is $-u$, but $u+(-u)=0$ is not a unit.  By the same token, if $n$ is any nilpotent element, then both $-u$ and $u+n$ are units but $n=(u+n)+(-u)$ is not. We are interested in rings which are as close as possible to having the units be closed under addition.  Namely, the following is the basic definition of our paper:

\begin{defn}\label{def:ua}
A commutative ring $R$ is \emph{\ua} if whenever $u,v$ are units of $R$, $u+v$ is either a unit of $R$ or a nilpotent element.
\end{defn}

This condition has surely come up in many contexts over time.  For instance, a polynomial ring over a field is \ua, but $\Z$ is not.  It has arisen in the work of the present authors at least twice: in a study of perinormality (a condition close to normality), see \cite[Theorem 4.2]{nmeSh-peripull}, and in a study of Euclidean and Egyptian domains, see \cite[Theorem 2.10]{nme-Euclidean}.  Earlier, in \cite{DJ-utriv} and \cite{HeiRo-utriv} people studied rings in which 1 is the only unit.  In the noncommutative literature \cite{Cal-UU, DaLa-UU}, there were studies of the more general ring-theoretic property (the \emph{UU} property) where every unit is unipotent -- i.e. the sum of a unit and a nilpotent element.  We show that such rings, when commutative, are necessarily \ua. However, we are not aware of Definition~\ref{def:ua} ever being made before. This article then represents our initial investigation of this topic which appears fundamental in commutative ring theory.  In it, we find connections with such diverse topics as monoid algebras, irreducible algebraic varieties, elliptic curves, and Euclidean domains.

It turns out that there are a number of equivalent criteria for unit-additivity (See Proposition~\ref{pr:ua}), including that the units and nilpotent elements form a subring of $R$.  Throughout Section~\ref{sec:ua}, we explore examples, nonexamples, and constructions of unit-additivity.  We show that unit-additivity imposes conditions on the characteristic, the nilradical, and the Jacobson radical of the ring.  We show that unit-additivity often passes to subrings, and occasionally to residue class rings, under many popular constructions.  
Additionally, we show that 
when a Euclidean domain $D$ is finitely generated over a field $k$, and either $k$ is algebraically closed or $1$ is the only unit of $D$, then either $D \cong k$ or $D \cong k[x]$
(see Theorem~\ref{thm:Eucl}).  We also show that one may ``reduce to the reduced case'' when studying unit-additivity (see Proposition~\ref{pr:uamodnilpotents}).

Section~\ref{sec:mon} is devoted to characterizing unit-additivity in monoid algebras in terms of properties of the monoid and of the base ring. 
 The main theorem of Section~\ref{sec:mon} gives equivalent criteria for monoid algebras to be unit-additive when the monoid is cancellative; see Theorem~\ref{thm:monalg}.  It says among other things that when $A$ is a commutative ring of characteristic zero and $M$ is a cancellative monoid, then $A[M]$ is \ua\ if and only if $A$ is \ua\ and $M$ is torsion-free and has no nonidentity invertible elements.  We obtain a particularly clean statement for affine semigroup algebras; see Proposition~\ref{pr:tfmonalg}.

The purpose of Section~\ref{sec:affine} is to give a criterion for unit-additivity of an affine domain $R$ over an algebraically closed field $k$ in terms of a restriction on $k$-algebra maps from $k[t,1/t]$ to $R$. See Theorem~\ref{thm:puncaone}. This translates to a geometric condition for an affine scheme to be the spectrum of a \ua\ ring.  See Theorem~\ref{thm:geom}, which says among other things that the coordinate ring of an irreducible affine variety over an algebraically closed field is \ua\ if and only if it admits no nonconstant morphisms to the punctured affine line.  Hence, unit-additive affine rings over an algebraically closed field can be seen as a natural setting for extensions of the Fundamental Theorem of Algebra; see Remark~\ref{rem:FTA}.

In Section~\ref{sec:elliptic}, we isolate a fairly general sort of equation, represented by a polynomial in two variables over a perfect field, which in degree 3 is the typical equation arising from the study of elliptic curves.  We show that all rings arising from these equations are \ua\ integral domains, which are usually Dedekind domains but seldom PIDs.

In Section~\ref{sec:uadim}, we introduce an invariant called \uad, defined by an iterative process, which measures how \emph{far} a domain is from being \ua.  We show that all \uad s can occur (see Examples~\ref{ex:uadimisdim} and \ref{ex:inftyua}), with all finite \uad s occuring in finitely generated algebras over an arbitrary field. In Theorem~\ref{thm:uadimbound}, we show that at least for affine algebras, \uad\ is bounded above by Krull dimension. However, the difference between these two dimensions can be made to be arbitrarily large (see Corollary~\ref{cor:udimvars}).

The short section~\ref{sec:uad1} is devoted to the study of domains of \uad\ 1 -- i.e. domains that are as close as possible to being \ua\ without actually satisfying the condition.  In particular, we give a sufficient condition (Proposition~\ref{pr:Jacudim}) in terms of a nonvanishing Jacobson radical, but then we show in  Example~\ref{ex:algint} that the condition is not necessary for a domain to have \uad\ 1.

In Section~\ref{sec:uac}, we introduce for any domain $R$ a localization of $R$ that is \ua\ and has a universal mapping property with respect to injective maps from $R$ to \ua\ domains (See Theorem~\ref{thm:ualoc}).  This then coincides with the \emph{\ua\ closure} of $R$ -- i.e., the intersection of all \ua\ subrings of the fraction field of $R$ that contain $R$ (see Proposition~\ref{pr:uacl}).

We conclude the paper with five interesting questions for further research; see Section~\ref{sec:questions}. 

Despite appearances from the summary above, most of the results about unit-additivity and \uad\ do not require the ring to be an integral domain.  However, we must often at least assume that the nilradical is prime. In a paper in prepartion, we largely remove these obstructions.

  \vspace{.1in}
\noindent \textbf{Notation and conventions:} For a ring $R$, the nilradical of $R$ is denoted $\cN(R)$.  The Jacobson radical is $\Jac(R)$. The group of units of $R$ is denoted $U(R)$.  All rings are commutative and unital.

\section{Unit additivity and \uty}\label{sec:ua}
In this section, we start with some characterizations of unit-additivity and give some constructions which preserve it, along with limiting counterexamples. Then we recall \uty\ and relate it to unit-additivity.  Finally we show that the study of unit-additivity reduces to a study of reduced rings.

\begin{prop}\label{pr:ua}
Let $R$ be a nonzero ring.  The following are equivalent: \begin{enumerate}
    \item\label{it:add1} For any unit $u$ of $R$, $u+1$ is either a unit or nilpotent.
    \item\label{it:ua} $R$ is \ua.
    \item\label{it:sr} The set $U(R) \cup \cN(R)$ is a subring of $R$.
    \item\label{it:0dimsr} The set $U(R) \cup \cN(R)$ is a zero-dimensional local subring of $R$.
    \item\label{it:list} For any finite list $u_1, \ldots, u_n$ of units of $R$, $\sum_{i=1}^n u_i$ is either a unit or nilpotent.
\end{enumerate}
If $R$ is a reduced ring, then the following condition is also equivalent to the above conditions: \begin{enumerate}
    \setcounter{enumi}{5}
    \item\label{it:subf} The set $U(R) \cup \{0\}$, with structure inherited from $R$, is a field, called the \emph{field of units} of $R$.
\end{enumerate}
\end{prop}

\begin{proof}
(\ref{it:0dimsr}) $\implies$ (\ref{it:sr}) $\implies$ (\ref{it:ua}) $\implies$ (\ref{it:add1}): clear from the definitions.

(\ref{it:add1}) $\implies$ (\ref{it:ua}): Let $f,g$ be units of $R$.  Then $u=f/g$ is a unit, so by assumption $u+1$ is either a unit or nilpotent.  Then since $g$ is a unit, if $u+1$ is a unit then so is $(u+1)g = f+g$, and if $u+1$ is nilpotent then since $\cN(R)$ is an ideal, $(u+1)g =f+g$ is also nilpotent.

(\ref{it:ua}) $\implies$ (\ref{it:sr}): Write $k=U(R) \cup \cN(R)$.  Let $a,b \in k$.  We have $1\in U(R) \subseteq k$ and $0 \in \cN(R) \subseteq k$. So let $a,b \in k$.  We must show that $a-b$ and $ab$ are in $k$.  There are three cases to consider.

Case 1: Suppose $a,b \in U(R)$.  Then $-b = (-1)b \in U(R)$ as well (since $-1 \in U(R)$) and since $R$ is \ua, we have $a-b=a+(-b) \in k$.  But also $ab \in U(R) \subseteq k$ since $U(R)$ is closed under multiplication.

Case 2: Suppose $a\in U(R)$ and $b \in \cN(R)$.  Then $ab \in \cN(R)$ (since $\cN(R)$ is an ideal) and $a-b \in U(R)$ (since the sum of a unit and a nilpotent is always a unit).  By symmetry, the same argument holds if $a\in \cN(R)$ and $b \in U(R)$.

Case 3: Suppose $a,b \in \cN(R)$.  Then $a-b, ab \in \cN(R)$ since $\cN(R)$ is an ideal.

(\ref{it:sr}) $\implies$ (\ref{it:0dimsr}): Write $k=U(R) \cup \cN(R)$.  Note that the units of $k$ and those of $R$ coincide.  Hence, the set of nonunits of $k$ consist of nilpotent elements.  It follows that $\cN(R)$ is the unique prime ideal of $k$.

(\ref{it:list}) $\implies$ (\ref{it:ua}): Specialize to the case $n=2$.

(\ref{it:ua}) $\implies$ (\ref{it:list}): Since as is well-known, a sum of two nilpotents is nilpotent and the sum of a unit and a nilpotent is a unit, it follows by a straightforward induction argument that (\ref{it:ua}) implies that any finite sum of units is either a unit or nilpotent.

Finally, when $R$ is reduced, (\ref{it:0dimsr}) $\iff$ (\ref{it:subf}) since $\cN(R) = \{0\}$, and the only zero-dimensional reduced local rings are fields.
\end{proof}

\begin{cor}\label{cor:char}
Let $R$ be a nonzero \ua\ ring.  If $\chr(R)>0$, then $\chr(R)$ is a  power of a prime number.
\end{cor}

\begin{proof}
Let $k = U(R) \cup \cN(R)$.  Then $k$ is a local subring of $R$.  But the only possible characteristics for a local ring are $0$ and powers of primes.  Hence, $\chr(R) = \chr(k) = p^n$.
\end{proof}

\begin{example}\label{ex:first}
For any ring $R$, $R$ is \ua\ if and only if $R[x]$ is \ua.  This is because $U(R[x]) = U(R)$ and $\cN(R[x]) \supseteq \cN(R)$.  For a careful proof and generalization, see Corollary~\ref{cor:sg}.
\end{example}

\begin{example}\label{ex:laurent}
  The ring $R[x, x^{-1}]$ (where $x$ is an indeterminate over $R$) is never \ua, since $x+x^{-1}$ is neither a unit nor nilpotent.
\end{example}

\begin{example}\label{ex:Z} As any reduced \ua\ ring contains a field, $\Z$ is not \ua.
\end{example}

\begin{example}\label{ex:Jacnil}
No ring such that $\Jac(R) \neq \cN(R)$ is \ua.  Indeed, let $j \in \Jac(R) \setminus \cN(R)$.  Then $1+j$ and $-1$ are units, but $j=(1+j)+(-1)$ is neither a unit nor nilpotent.  Hence, any local ring of positive dimension, and any ring of the form $R[\![x]\!]$ (with $x$ an analytic indeterminate), is not \ua. See also Proposition~\ref{pr:Jacudim} and Corollary~\ref{cor:Jacudim}.
\end{example}

\begin{prop}\label{pr:subalg}
Let $S$ be \ua.  Let $k = U(S) \cup \cN(S)$.  Let $R$ be a $k$-subalgebra of $S$.  Then $R$ is \ua.
\end{prop}

\begin{proof}
    Let $u,v \in U(R)$.  Since $U(R) \subseteq U(S)$, it follows from the fact that $S$ is \ua\ that $u+v \in \cN(S) \cup U(S) = k$.  If $u+v \in \cN(S)$ then it is nilpotent and we are done.  Otherwise $u+v$ is a unit in $k$, so $w = (u+v)^{-1} \in k \subseteq R$.  Hence $u+v$ is a unit of $R$.
\end{proof}

\begin{prop}\label{pr:intext}
    Let $R \subseteq S$ be an integral extension of rings. If $S$ is \ua, then so is $R$.
\end{prop}

\begin{proof}
Let $u,v$ be units of $R$ such that $u+v \notin \cN(R)$.  Then $u+v \notin \cN(S)$, so $u+v \in U(S) \cap R = U(R)$, by the Lying-Over property of integral extensions (see e.g. \cite[Theorem 9.3(i)]{Mats}).
\end{proof}

\begin{example}
If $S$ is merely \emph{almost integral} over $R$, then Proposition~\ref{pr:intext} may fail.  For example, let $S = \Q[x]$ and $R = \Z + x\Q[x]$.  Then $S$ is \ua\ and almost integral over $R$ (since for any $f\in S$, we have $xf^n \in R$ for all $n\in \N$), but $R$ is not \ua\ since 2 is not a unit of $R$.
\end{example}

\begin{example}\label{ex:integralnonascent}
    The converse to Proposition~\ref{pr:intext} can fail, even when $R$, $S$ are integral domains finitely generated over the same field.  Let $S = K[X,Y]/(XY-1)$, where $k$ is a field over which $X,Y$ are independent indeterminates.  Let $x,y$ denote the images of $X,Y$ (respectively) in $S$, and consider the subring $R=k[x+y]$ of $S$.  Note that $x+y$ is transcendental over $k$, so that $k[x+y]$ is isomorphic to a polynomial ring in one variable over $k$.  Hence by Example~\ref{ex:first}, $R$ is \ua, whereas $S$ is \emph{not} \ua\ due to Example~\ref{ex:laurent}.  However, $S$ is integral over $R$.  To see this, note that $x, y$ are integral over $R$ via the equations $y^2 - (x+y)y+1=x^2 - (x+y)x+1=0$.
\end{example}

Next, consider the following definition, closely related to that of unit-additivity, which by Proposition~\ref{pr:utriv}(2) coincides with the condition investigated in \cite{HeiRo-utriv} and \cite{DJ-utriv} for reduced commutative rings.

\begin{defn}\label{def:utriv}\cite{Cal-UU}
A ring $R$ is \emph{UU} (units are unipotent) if $U(R) = 1+\cN(R)$.
\end{defn}

\begin{prop}\label{pr:utriv}
Let $R$ be a nonzero commutative ring.
\begin{enumerate}
    \item \cite[Theorem 2.6(1)]{DaLa-UU} If $R$ is \ut, then $\chr(R)$ is a power of $2$.
    \item If $R$ is \ut, then $R$ is \ua.
    \item \cite[Remark 2.1(1)]{Cal-UU} Suppose $R$ is reduced.  Then $R$ is \ut\ $\iff U(R) = \{1\}$.  In this case $\chr(R) = 2$.
\end{enumerate}
\end{prop}

\begin{proof}
We need only prove (2).  Let $u,v$ be units of $R$.  We have $u=1+n$, $v=1+m$, with $n,m \in \cN(R)$.  Then $u+v = 2 \cdot 1_R +m+n \in \cN(R)$, since by (1), $\chr R$ is a power of $2$.  Hence, $R$ is \ua.
\end{proof}

It turns out that unit-additivity and \uty\ are intimately connected in any nontrivial product of rings:

\begin{prop}\label{pr:product}
    Let $R=S \times T$, where $S,T$ are nonzero commutative rings. The following are equivalent: \begin{enumerate}
        \item\label{it:prRua} $R$ is \ua.
        \item\label{it:prcomps} $S$ and $T$ are \ut.
        \item\label{it:prRut} $R$ is \ut.
    \end{enumerate}
\end{prop}

\begin{proof}
The equivalence of (\ref{it:prcomps}) and (\ref{it:prRut}) is \cite[(2.2)]{DaLa-UU}.  The implication (\ref{it:prRut}) $\implies$ (\ref{it:prRua}) follows from Proposition~\ref{pr:utriv}(2).

(\ref{it:prRua}) $\implies$ (\ref{it:prcomps}):  Recall that $U(R) = U(S) \times U(T)$. Let $u \in U(S)$.  Then $(u,1), (-1,-1) \in U(R)$, so $(u-1,0) = (-1,-1) + (u,1) \in U(R) \cup \cN(R)$.  But it cannot be a unit since $0 \notin U(T)$.  Hence it must be nilpotent, so $u-1$ must be nilpotent, i.e., $u \in 1+\cN(S)$.  Thus, $S$ is \ut\ since $u \in U(S)$ was arbitrary.  By symmetry, $T$ is also \ut.
\end{proof}

It is easy to see that that the above argument is valid for an arbitrary product of two or more rings.   Hence, a factor ring of a \ua\ ring is not \ua\ in general.  However it is true for a certain class of rings:

\begin{example}\label{ex:Boolean}
Recall that a ring $R$ is called \emph{Boolean} if $x^2 = x$ for all $x\in R$. Any Boolean ring is \ut\ \cite[Theorem 4.1]{DaLa-UU}.  Moreover, any homomorphic image of a Boolean ring is Boolean, hence \ut\ and thus \ua.
\end{example}

It is natural to wonder what Euclidean domains are \ua.  For any field $k$ and indeterminate $X$ over $k$, the rings $k$ and $k[X]$ are \ua\ Euclidean domains.  The following is a partial converse that depends on the main theorems of two separate papers.

\begin{thm}\label{thm:Eucl}
Let $R$ be a \ua\ Euclidean domain.  Let $k$ be a maximal subfield of $R$, and assume $R$ is finitely generated as a $k$-algebra.  Suppose either \begin{enumerate}
    \item $k=\F_2$, or
    \item $k$ is algebraically closed.
\end{enumerate}
Then either $R \cong k$ or $R \cong k[X]$, where $X$ is an indeterminate over $k$.
\end{thm}

\begin{proof}
First suppose $k=\F_2$.  By unit-additivity, we have $U(R) = k \setminus \{0\} = \{1\}$, so that $R$ is \ut.  Then by \cite[Theorem 3.5]{HeiRo-utriv}, $R \cong k$ or $R \cong k[X]$.

Next suppose $k$ is algebraically closed.  By \cite[Theorem 1.1(A)]{Br-Euaffine}, if $R \neq k$, then $\Spec(R)$ is isomorphic to an affine open subscheme of $\PP^1_k$.  Since $k$ is algebraically closed, for any closed point $p$ of $\PP^1_k$ we have $\PP^1_k \setminus \{p\} \cong \A^1_k$.  The only open subsets of $\PP^1_k$ consist of the empty set (corresponding to the zero ring) and the complement of a finite set of closed points in $\PP^1_k$.  But since $\PP^1_k$ is not affine, the affine open subsets must then be isomorphic to removing at least one point of $\PP^1_k$, hence an affine open subset of $\A^1_k$.  But $\A^1_k = \Spec k[X]$, whereas the complement of any finite nonempty set of points of $\A^1_k$ is isomorphic to $\Spec k[X, 1/f]$ for some nonconstant $f\in k[X]$.  Let $c$ be a root of $f$, and consider the dominant $k$-scheme morphism morphism $X = \Spec k[X, 1/f] = \A^1_k \setminus \{$the roots of $f\} \rightarrow \A^1_k \setminus \{0\}$ given by $\lambda \mapsto \lambda-c$. 
 Then by Theorem~\ref{thm:geom} below, $k[X,1/f]$ is not \ua.
\end{proof}

\begin{rem}\label{rem:Schwede}
The above proof does not work when $k$ is not algebraically closed, for in that case, the complement of a closed point $p$ need not be isomorphic to $\A^1_k$, particularly if $p$ is not $k$-rational.  Thus, proper affine open subsets of $\PP^1_k$ need not be spectra of localizations of $k[X]$. For instance, let $k=\R$ and let $p$ be the closed point corresponding to the irreducible homogeneous polynomial $x^2 + y^2$. Then the affine coordinate ring of $\PP^1 \setminus \{p\}$ is the degree zero subring of $\R\left[x,y, \frac 1{x^2 + y^2}\right]$, which can be expressed as $\R\left[\frac {x^2}{x^2+y^2}, \frac{xy}{x^2 + y^2}\right] \cong \R[s,t] / (s^2 - s+t^2)$.  This is also not a counterexample, since the latter is not Euclidean, as it is not even a PID, but in any case we don't have the tools at this juncture to answer the question.
\end{rem}

In the final results of this section, we show that we may in many cases ``reduce to the reduced case'' in questions of unit-additivity, which echoes the corresponding statements for \uty\ (see \cite[Theorem 2.4]{DaLa-UU}).

\begin{lemma}\label{lem:liftunitsnilpotents}
    Let $R$ be a commutative ring and $I$ an ideal with $I \subseteq \cN(R)$.  Let $A=R/I$.  Let $x\in R$, and let $\bar x$ be its image in $A$.  Then $x \in \cN(R)$ (resp. $x\in U(R)$) if and only if $\bar x \in \cN(A)$ (resp. $\bar x\in U(A)$).
\end{lemma}

\begin{proof}
   The result follows from the well known facts that $u\in R$ is a unit if an only $u+x$ is a unit for any $x\in \cN(R)$ and that $x\in A$ is nilpotent if and only if it is the residue of a nilpotent element of $R$.
\end{proof}

\begin{prop}\label{pr:uamodnilpotents}
Let $R$ be a commutative ring and $I$ an ideal of $R$ with $I \subseteq \cN(R)$.  Then $R$ is \ua\ 
$\iff R/I$ is \ua.
\end{prop}

\begin{proof}
Set $A := R/I$.
Suppose $R$ is \ua\ and let $x,y \in R$ such that $\bar x, \bar y \in U(A)$.  Then by Lemma~\ref{lem:liftunitsnilpotents}, $x,y \in U(R)$.  Thus, $x+y \in U(R) \cup \cN(R)$, so that by Lemma~\ref{lem:liftunitsnilpotents} again, $\bar x + \bar y = \overline{x+y} \in U(A) \cup \cN(A)$.

Conversely suppose $A$ is \ua\ and let $x,y \in U(R)$.  Then by Lemma~\ref{lem:liftunitsnilpotents}, $\bar x, \bar y \in U(A)$, so that $\overline{x+y} = \bar x + \bar y \in U(A) \cup \cN(A)$, so that by Lemma~\ref{lem:liftunitsnilpotents} again, $x+y \in U(R) \cup \cN(R)$.
\end{proof}

\begin{cor}
Let $A$ be a ring, $M$ an $R$-module, and $R=A(+)M$ the idealization of $M$ $($see \cite[p. 2]{NagLR}; also known as a \emph{trivial extension} of $A$ \cite{AnWi-idealization}$)$.  Then $A$ is \ua\ 
$\iff R$ is \ua.
\end{cor}
\begin{proof}
We have $R/M \cong A$ and $M^2=0$ in $R$.
\end{proof}

\section{Monoid graded rings and monoid algebras}\label{sec:mon}
In this section, $M = (M, \cdot, 1)$ is a commutative monoid.  Our goal is to characterize when $A[M]$ is unit-additive in terms of properties of $A$ and $M$.
Recall the following from, e.g., \cite[Chapters 2 and 4]{BrGu-polybook}.
\begin{defn}
Given a commutative monoid $(M, \cdot,1)$, \begin{itemize}
    \item We say $x\in M$ is a \emph{unit} if there is some $y \in M$ with $xy=1$.
    \item Let $U(M)$ denote the set (actually a group) of units of $M$.
    \item We say $M$ is \emph{positive} if $1$ is its only unit.
    \item We say $M$ is \emph{cancellative} if whenever $x,y,z \in M$ with $xz=yz$, we have $x=y$.
    \item Any element $x\in M$ that admits $n \in \N_{>0}$ such that $x^n=1$ is called a \emph{torsion} element, and the smallest such $n$ is called the \emph{order} of $x$.   If $1$ is the only torsion element of $M$, we say $M$ is \emph{torsion-free}.
    \item The set of torsion elements $T=T(M)$ form a submonoid that is a group, called the \emph{torsion subgroup} of $M$.
    \item For any prime number $p$, the set $M_p$ of elements of $M$ that are torsion elements of order $p^n$ for some $n \in \N_{\geq 0}$ is a subgroup of $T(M)$, called the \emph{$p$-torsion subgroup} of $M$.
    \item An \emph{$M$-graded ring} is a ring $R$ along with a direct sum decomposition of additive groups $R = \bigoplus_{x \in M} R_x$, subject to the rule $R_x R_y \subseteq R_{xy}$ for all $x, y \in M$ -- i.e. whenever $r \in R_x$ and $s \in R_y$, we have $rs \in R_{xy}$.
\end{itemize}
\end{defn}
It follows that if $R$ is an $M$-graded ring, then $R_1$ is a subring of $R$ and each $R_x$ is an $R_1$-module.

\begin{prop}\label{pr:graded}
Let $M$ be a commutative monoid, and let $R$ be an $M$-graded ring.  If $R$ is \ua\ then so is $R_1$, where $R_1$ is the graded component of $R$ corresponding to the identity $1\in M$.  The converse holds provided $M$ is cancellative, torsion-free, and positive.
\end{prop}

\begin{proof}
Suppose $R$ is \ua. Let $u,v \in R_1$ be units.  Then since $u,v$ are units of $R$, $u+v$ is either nilpotent, or a unit in $R$. If $u+v$ is a unit in $R$, let $w\in R$ with $(u+v)w=1$.  Then the $1$st graded component of the right hand side is $1$, whereas the $1$st graded component of the left hand side is $(u+v)w_1$, where $w_1$ is the $1$st graded component of $w$. Thus, $u+v$ is a unit in $R_1$.  It follows that $R_1$ is \ua.

Conversely suppose that $R_1$ is \ua\ and $M$ is  cancellative, torsion-free, and positive. Let $u,v$ be units of $R$.  Then by \cite[Proposition 4.10(a)]{BrGu-polybook}, we have $u=u_1 + x$ and $v=v_1 + y$, where $u_1,v_1 \in U(R_1)$ and $x,y$ are nilpotent.  We have $u+v = (u_1 + v_1) + (x+y)$, and since $R_1$ is \ua, $u_1+v_1$ is either a unit or nilpotent.  Since $x+y$ is nilpotent, it follows that $u+v$ is a unit (resp. nilpotent) if and only if $u_1 + v_1$ is a unit (resp. nilpotent).  Hence, $u+v$ is either a unit or nilpotent, so $R$ is \ua.
\end{proof}

\begin{cor}\label{cor:sg}
Let $R = A[x]$, or $A[x_1, \ldots, x_n]$, or $A[x_1, x_2, \ldots]$, or $A[M]$ where $M$ is a subsemigroup of the positive orthant of $\R^n$.  Then $R$ is \ua\ iff $A$ is \ua.
\end{cor}

When $A$ has prime nilradical, we will see a generalization of Corollary~\ref{cor:sg} in Proposition~\ref{pr:udimsubvars}, after the introduction of \emph{\uad}.

Next we pass to monoid algebras.  Given a ring $A$ and a multiplicative monoid $M$, we construct (see e.g. \cite[Section 4.B]{BrGu-polybook}) an $M$-graded $A$-algebra $A[M]$ in the following way:  As an $A$-module, $A[M]$ is free on the basis $M$, with $x$-graded component $A[M]_x = Ax \cong A$ as $A$-modules.  The multiplication is defined on ``monomials'' by setting $(ax) \cdot (by) := (ab)(xy)$ for $a,b \in A$ and $x,y \in M$, where the product $ab$ is as in $A$, and the product $xy$ is as in $M$.  Then one defines the rest of the multiplication on $A[M]$ by the associative law.  Note that with these rules, it follows that $1 = 1_{A[M]} = 1_A 1_M$.

First we handle the case of monoid algebras where the monoid is both cancellative and torsion-free.

\begin{prop}\label{pr:tfmonalg}
Let $A$ be a commutative ring, and let $M$ be a commutative, cancellative, torsion-free monoid.  The following are equivalent: \begin{enumerate}
    \item $R := A[M]$ is \ua.
    \item $A$ is \ua\ and $M$ is positive.
\end{enumerate}
\end{prop}

\begin{proof}
(2) $\implies$ (1): This follows from Proposition~\ref{pr:graded}.

(1) $\implies$ (2): By Proposition~\ref{pr:graded}, we need only show that $M$ is positive.  For this, first assume that $A$ is reduced.  Either $A=B\times C$ for nontrivial reduced rings $B,C$, or $\Spec A$ is connected.

In the first case, we have $A[M] \cong B[M] \times C[M]$, so by Proposition~\ref{pr:product}, $U(B[M]) = \{1\}$.  But any unit of $M$ is a unit of $B[M]$. Thus $U(M) = \{1_M\}$. That is, $M$ is positive.

In the second case, by \cite[Proposition 4.10(b)]{BrGu-polybook}, every unit of $A[M]$ is homogeneous -- i.e., every unit is of the form $ax$, with $a\in A$ and $x\in M$.  Accordingly, let $x\in M$ be a unit of $M$, hence a unit of $A[M]$.  Since $A[M]$ is \ua\ and reduced, $1-x$ must be a unit or 0.  But if it is 0, then $x=1$, and if it is a unit, then we must have $\deg_M(x) = 1$, whence $x=1$ again.  Thus, $M$ is positive.

Finally we drop the assumption that $A$ is reduced.  It is clear that $\cN(A)[M] := \bigoplus_{x\in M} \cN(A)x \subseteq \cN(A[M])$.  On the other hand, $A[M] / \cN(A)[M] \cong (A/\cN(A))[M]$, which is reduced by \cite[Theorem 4.19]{BrGu-polybook}.  Therefore $\cN(A[M]) \subseteq \cN(A)[M]$, whence $\cN(A[M]) = \cN(A)[M]$.  But by Proposition~\ref{pr:uamodnilpotents}, $(A/\cN(A))[M] \cong A[M] / \cN(A)[M] = A[M] / \cN(A[M])$ is \ua.  Then by implication (1) $\implies$ (2) in the reduced case, which we have shown above, it follows that $M$ is positive.
\end{proof}

For the main theorem of this section, we dispense with the condition of torsion-freeness.  First, though, we need a lemma on torsion subgroups of monoids.

\begin{lemma}\label{lem:torsion}
Let $(M,\cdot, 1)$ be a commutative monoid and let $G$ be a subgroup of $T(M)$.  Then $M/G$, with elements given by the cosets $\{\bar m =m+G \mid m \in M\}$ and addition given by $\bar m + \bar n = \overline{m+n}$, is a monoid.  Moreover, if $G=T(M)$, then $M/G$ is torsion-free, and if $G=M_p$ for some prime number $p$, then $M/G$ is $p$-torsion-free.
\end{lemma}

\begin{proof}
The first statement follows from \cite[Theorem 4.4]{Gil-CSR}.  The rest of the statements have the same proofs as the analogous statements for $M$ an abelian group.
\end{proof}

\begin{thm}\label{thm:monalg}
Let $A$ be a nonzero commutative ring and $M$ a cancellative commutative monoid.  Then $A[M]$ is \ua\ if and only if either \begin{enumerate}
    \item $\chr(A)=0$, $A$ is \ua, and $M$ is torsion-free and positive, or
    \item $\chr(A)$ is a positive power of a prime number $p$, $A$ is \ua, and $M/M_p$ is torsion-free and positive.
\end{enumerate}
\end{thm}

\begin{proof}
First suppose $A[M]$ is \ua. Then by Proposition~\ref{pr:graded}, so is $A$.  Thus by Corollary~\ref{cor:char}, $\chr (A)$ is either $0$ or a power of a prime.

Suppose $\chr(A)=0$. If $M$ has nontrivial torsion subgroup $T(M)$, then there is some prime number $q$ and some $t\in T(M)$ of order $q$.  That is, $H = \{t^j \mid 0 \leq j < q\}$ is a subgroup of $T(M)$ of order $q$.  Since $q=q \cdot 1_A$ is not nilpotent (since $\chr(A) =0$), then since $A$ is \ua\ we have that $q$ is a unit of $A$.  Then by \cite[Theorem 10.1]{Gil-CSR}, the element $e=q^{-1} \cdot \sum_{j=0}^{q-1} t^j$ is a nontrivial idempotent of $A[M]$.  On the other hand, since each $t^j$ is a unit and $A[M]$ is \ua, it follows that $qe$, hence also $e$, is either a unit or nilpotent.  But the only nilpotent idempotent in any ring is 0, and the only idempotent unit is 1, so $e$ is a trivial idempotent, a contradiction.

Suppose on the other hand that $\chr(A) = p^n$. By replacing $A$ by $A/pA$ (using Proposition~\ref{pr:uamodnilpotents}), we may assume $\chr(A) = p$.  Now suppose there is some $t \in T(M)$ with order $q$, where $q$ is a prime number other than $p$.  Then since $A$ contains $\F_p$ and $p \nmid q$, $q$ is a unit in $A$.  Then by \cite[Theorem 10.1]{Gil-CSR}, the element $e=q^{-1} \cdot \sum_{j=0}^{q-1} t^j$ is a nontrivial idempotent of $A[M]$.  But this is impossible by the argument we gave in the characteristic zero case.
Hence, every element of $T(M)$ has order a power of $p$.  That is, $T(M) = M_p$, so that by Lemma~\ref{lem:torsion}, $M/M_p$ is torsion-free.  Then by Proposition~\ref{pr:tfmonalg}, $M/M_p$ is positive.

Conversely, suppose condition (1) holds.  Then by Proposition~\ref{pr:tfmonalg}, $A[M]$ is \ua.  

Finally, suppose condition (2) holds. By Proposition~\ref{pr:tfmonalg},   $(A/pA)[M/M_p]$ is \ua.  Note that $(A/pA)[M/M_p] \cong (A/pA)[M] / (\{m-1 \mid m \in M_p\})$.  But the latter ideal is nilpotent.  To see this, let $m \in M_p$.  Then there is some $n$ with $0 = m^{p^n} - 1 = (m-1)^{p^n}$ since $\chr(A/pA) = p$.  Thus by Proposition~\ref{pr:uamodnilpotents}, $(A/pA)[M]$ is \ua. But $(A/pA)[M] \cong A[M] / pA[M]$, so that since $p$ is nilpotent in $A$, another application of Proposition~\ref{pr:uamodnilpotents} shows that $A[M]$ is \ua.
\end{proof}

\begin{cor}\label{cor:groupring}
Let $A$ be a nonzero commutative ring and $G$ an abelian group.  Then $A[G]$ is unit-additive if and only if $A$ is \ua\ and either \begin{enumerate}
    \item $\chr A=0$, and $G$ is trivial, or 
    \item $\chr A$ is a positive power of a prime number $p$, and $G$ is $p$-torsion.
\end{enumerate}
\end{cor}

\section{Affine algebras}\label{sec:affine}
The goal of this section is to find algebraic and geometric characterizations of when the coordinate ring of an affine variety is \ua.

\begin{thm}\label{thm:puncaone}
    Let $R$ be an integral domain that is finitely generated over a field $k$.  Let $L$ be the integral closure of $k$ in $R$. Let $A=k[t,t^{-1}]$, where $t$ is an indeterminate over $k$. Then $L$ is the unique maximal subfield of $R$ containing $k$.  Moreover, exactly one of the following is true: \begin{enumerate}
          \item $R$ is \ua, and the only $k$-algebra homomorphisms from $A$ to $R$ are the ones that send $t \mapsto \alpha$ for some $\alpha \in L^\times$ (i.e., $U(R) = L^\times$).  None of these are injective.
          \item $R$ is not \ua, and there is an injective $k$-algebra map $\phi:A \ra R$. 
    \end{enumerate}
\end{thm}

\begin{proof}
For the first statement, first note that $L$ is a field by \cite[Theorem 16]{Kap-CR}.  Now let $F$ be a field with $k \subseteq F \subseteq R$.  By  Zariski's lemma (\cite[Exercise 15 of Section 1-3]{Kap-CR} or \cite[$\mathrm{H}^n_3$ on p. 363]{Zar-newnull}), if $\m$ is a maximal ideal of $R$, $R/\m$ is a finite extension field of $k$.  But the composite $k \ra  F \ra R \ra R/\m$ embeds $F$ as a $k$-subalgebra of $R/\m$; hence $F$ is a finite extension field of $k$, whence integral. Thus, $F \subseteq L$.

Next, suppose $R$ is \ua.  Clearly all the elements of $L^\times$ are units of $R$.  Conversely, let $u$ be a unit of $R$.  Then by unit-additivity, $k[u]$ is a subfield of $R$.  Then as above, we have a tower of fields $k \hookrightarrow k[u] \hookrightarrow R/\m$, so that $k[u]$ is finite, hence integral, over $k$.  Thus, $u \in L$, so that $L$ is the field of units of $R$.  Now let $\phi: A \ra R$ be a $k$-algebra map.  Let $\alpha = \phi(t)$.  Then $\phi(t^{-1})\alpha = \phi(t^{-1})\phi(t) = \phi(t^{-1}t) = \phi(1) = 1$.  Thus, $\alpha$ is a unit of $R$, so $\alpha \in L^\times$, whence $\alpha$ is integral over $k$.  Let $g(t) \in k[t]$ be the minimal monic polynomial of $\alpha$ over $k$.  Then $g(t) \neq 0$ in $A$, but $\phi(g(t)) = g(\alpha) = 0$, so $\phi$ is not injective.

Suppose on the other hand that $R$ is not \ua.   Then there is some unit $u$ of $R$ such that $u+1$ is a nonzero nonunit.  Hence, $u \notin L$.  Define $\gamma: k[t] \ra R$ to be the unique $k$-algebra map that sends $t \mapsto u$.  Since $u$ is a unit, $\gamma$ extends uniquely to a $k$-algebra map $\phi: k[t]_t=k[t,t^{-1}] \ra R$.  Moreover, $\phi$ is injective.  To see this, suppose $\ker \phi \neq 0$ and choose $g \in \ker \phi \setminus \{0\}$.
Then there is some $m \in \Z$ and some $c \in k^\times$ such that $h = ct^mg$ is a monic polynomial in $k[t]$.  We have $h(u) = (ct^mg)(u) = cu^m g(u) = cu^m \phi(g) = 0$.  Thus, $u$ is integral over $k$, so that $u\in L$, which is a contradiction.  Therefore, $\ker \phi=0$.
\end{proof}

We now have the setup for the geometric characterization of unit-additivity in finitely generated algebras over fields.

\begin{thm}\label{thm:geom}
Let $X=\Spec R$ be the scheme associated to an irreducible variety over an algebraically closed field $k$.  Let $C = \A^1_k \setminus \{0\}$.  Then exactly one of the following is true: \begin{enumerate}
    \item $R$ is \ua, and all $k$-scheme morphisms $X \ra C$ are constant.
    \item $R$ is not \ua, and there is a dominant $k$-scheme morphism $X \ra C$.  Moreover, any such map has cofinite image. 
\end{enumerate}
\end{thm}

\begin{proof}
After translating Theorem~\ref{thm:puncaone} into geometric language, it remains to show that any dominant map $X \ra C$ has cofinite image.  By \cite[Proposition II.2.6]{Hart-AG}, we may go back and forth between the scheme-theoretic viewpoint and the more set-theoretic viewpoint of varieties over $k$ with impunity in this setting.

By Chevalley’s
constructibility theorem \cite[(6.E)]{Hart-AG}, the image $f(X)$ of $X$ in $C$ is constructible.
By \cite[(6.C)]{Hart-AG}, since $f$ is dominant, $f(X)$ contains a nonempty
open subset $U$ of $C$, and hence $C \setminus f(X) \subseteq C \setminus U$ is finite.
\end{proof}

\begin{example}\label{ex:cofinite}
It is natural to wonder whether all cofinite subsets of $\puncaone$ are possible images in the maps from Theorem~\ref{thm:geom}(2). In fact they are.  To see this, let $\{\alpha_1, \ldots, \alpha_s\}$ be an arbitrary finite subset of $\puncaone$, and let $f = \prod_{i=1}^s (t-\alpha_i)$. Let $R = k[t, \frac 1 {tf}]$, and let $\phi:A \ra R$ be the inclusion map.  Then $\Max R \cong {\mathbb A}^1_k \setminus \{ {\mathbf 0}, \alpha_1, \ldots, \alpha_s\}$, and the corresponding map $\Max R \ra \Max A$ amounts to the inclusion map.
\end{example}

\begin{rem}\label{rem:FTA}
The Fundamental Theorem of Algebra says that any polynomial map $\C \ra \C$ has a root.  If one varies the \emph{source} of such maps, one can ask: For which irreducible algebraic sets $X$ does it hold that every polynomial map $X \ra k$ has a root?  In that vein, the above results yield the following:

Let $X$ be an irreducible algebraic set over an algebraically closed field $k$, and let $R = \Gamma(X)$ be its coordinate ring.  Then $R$ is \ua\ if and only if every polynomial map $X \ra k$ has a root.
That is, $R$ is \ua\ if and only if polynomial maps on $X$ satisfy the Fundamental Theorem of Algebra.

To see this, first suppose $R$ is not \ua.  By Theorem~\ref{thm:geom}(2), there is a dominant map $\phi: X \ra k \setminus \{0\}$.  Composing with the inclusion $k \setminus \{0\} \into k$ yields a nonvanishing nonconstant polynomial map on $X$.
On the other hand, suppose $R$ is \ua, and let $g: X \ra k$ be a nonconstant polynomial map. Then $g \in R \setminus k$, so by Theorem~\ref{thm:puncaone}(1), $g$ is a nonunit of $R$.  That is, $(g)$ is not the unit ideal, so by the Nullstellensatz, $V(g) \neq \emptyset$.
\end{rem}

\section{Equations arising from elliptic curves}\label{sec:elliptic}
It is well-known \cite[Proposition III.3.1]{Sil-arithbook} that up to isomorphism (with identity on the line at infinity) the affine variety associated to an elliptic curve over a perfect field $k$ has equation \begin{equation}\label{eq:ell}
f =f(X,Y)= Y^2 + c(X)Y + p(X), 
\end{equation} with $c \in k \oplus kX$ and $-p$ is a monic cubic.
If $\chr k \neq 2$, one may replace $Y$ by $Y-\frac{c(X)}2$ to obtain the more familiar form $f=Y^2 + p(X)$.  In the following, we let $p$ be an arbitrary polynomial of odd degree except where otherwise stated.  Elliptic curves are of basic importance in algebraic geometry, coding theory, and algebraic number theory.  We show below that the affine coordinate ring of an elliptic curve is always \ua, but seldom a PID.

Much of the development below is derived from \cite{HeiRo-utriv}, as we are generalizing  \cite[Theorem 2.1]{HeiRo-utriv} in several different directions.

\begin{lemma}\label{lem:irreducible} Let $k$ be a field and let $f= Y^2+c(X)Y+p(X)$, where $p(X) \in k[X]$ is a polynomial of odd degree and $c(X) \in k[X]$ of degree at most one.  If either $c(X)$ is a constant or $p(X)$ is of degree at least three, then $f$ is an irreducible element of $k[X,Y]$.  In particular $R = k[X,Y]/(f)$ is an integral domain.
\end{lemma}
\begin{proof} If $f$ is not irreducible, then it has a factor that is a polynomial in $Y$ (with coefficients in $k[X]$) of degree 2 or less.   Since the leading coefficient of $f$ is 1, any factor of $f$ has leading coefficient a unit of $k[X]$, namely is in $k$.  Thus if $f$ factors into a pair of nonunits, both must be of the form $dY+q$ with $d\in k \setminus \{0\}$ and $q \in k[X]$.  In particular there exists $\alpha \in k[X]$ that is a root of $f$.  If $c$ is constant, then $\alpha^2+c\alpha +p = 0$, which contradicts the assumption that deg $p$ is odd.

Now suppose that $c(X)$ has degree one and $p$ has degree at least three.   If def $\alpha \leq 1$, this is clearly a contradiction to the equation $\alpha^2+c\alpha +p = 0$, since deg $p > 2$.   If deg $\alpha > 1$, then since $p$ has odd degree and since deg $\alpha^2 > $ deg $cX$, we also reach contradiction.  Thus $f$ is irreducible and the concluding statement is clear.
\end{proof}

\begin{thm}\label{thm:ua} Let $k$ be a field and let $f\in k[X,Y]$ be as in Lemma~\ref{lem:irreducible}.  Then $R= k[x,y] =k[X,Y]/(f)$ is a \ua\ integral domain.  However, if $\deg p \geq 3$ and if $p$ has a root in $k$, then $R$ is not a PID.
\end{thm}
\begin{proof} By Lemma~\ref{lem:irreducible} we know that $R$ is a domain.  Note that we can view the polynomial ring $k[x]$ as a subring of $R$. Then using the substitution $y^2 =  -cy -p$, every element of $R$ can be written in the form $a+yb$, where $a,b \in k[x]$ are uniquely determined. 
Next, we see that that the fraction field $k(x,y)$ of $R$ is the splitting field of the polynomial $Z^2 + cZ + p \in k(x)[Z]$; its roots are $y$ and $p/y$.  The polynomial is then seen to be separable since otherwise $y=p/y$, so that $p=y^2 = -cy-p$, contradicting the uniqueness of representation shown above.  Hence the field extension $k(x,y)/k(x)$ is Galois.
Let $N$ be the norm of this extension.   We claim that for $u=a+by \in R$, we have $N(u) = a^2-abc+b^2p \in k[x]$.  Recall that $N(u) = \sigma_1(u)\sigma_2(u)$, where $\sigma_i$, $i=1,2$ are the elements of the Galois group of $k(x,y) $ over $k(x)$.  Since $y$ and $p/y$ are the roots of $f$, we have $N(u) = (a+by)(a+b(p/y))= a^2 +ab(y+(p/y))+b^2p$.  However, in $k(x,y)$, $y + (p/y) = (y^2+p)/y = -cy/y=-c$, from which we deduce that $N(u) = a^2-abc+b^2p$ as claimed.

Next we prove unit-additivity.   Let $u= a+by$ be a unit of $R$. Then $N(u)$ is a unit of $k[x]$, so $a^2-abc+b^2p\in k\setminus \{0\}$.  

First we assume that $c$ is a constant, i.e., $c$ has degree zero. If deg $a>$ deg $b$, then, since $p$ has odd degree, deg $(a^2-abc+b^2p) = \max\{$deg $a^2$, deg $b^2p\}$.   Conversely, if deg $b \geq$ deg $a$, then deg ($a^2-ab+b^2p$) = deg $b^2p$.  In either case both $a$ and $b^2p$ have degree zero.  Since $p$ has positive degree, we must have $b=0$ and $a \in k$ which means that $u =a \in k.$  Hence $R$ is \ua\ with field of units $k$.

Next we assume that $c$ has degree one and $p$ has degree at least three.
We break into cases:  First suppose deg $a \leq $ deg $b$; then, since deg $p \geq 3$, deg$(a^2-abc+b^2p) =$ deg $b^2p$.  Thus $b = 0$, so as before $u=a\in k$

Finally suppose that deg $a >$ deg $b$; then deg $a^2 \geq$ deg $abc$.
   If the inequality is strict then, because $p$ is of odd degree, one of deg $a^2$ or deg $b^2p$ is larger than the other, so
   deg$(a^2-abc+b^2p)  = \max \{$deg  $a^2$, deg $b^2p\}$.  Thus we are in the earlier case and $u\in k$.
Finally assume that deg $a^2 =$ deg $abc$.  Then deg  $a  = 1+$ deg $b$.  But, since deg $p \geq 3$, we have, if $b\neq 0$, deg $b^2p = $2deg $b$ + deg $p  > $ 2deg $a =$ deg $a^2$.  Thus we have deg$(a^2-abx+b^2p) =$ deg $b^2p \neq 0$. This contradiction again implies $b=0 $ and so $u\in k$, whence $R$ is unit-additive.

For the PID statement, suppose $\lambda \in k$ with $p(\lambda)=0$ and $\deg p \geq 3$.  Then $f \in (X-\lambda, Y)k[X,Y]$, so that $\m= (x-\lambda,y)$ is a proper ideal of $R$.  Suppose $\m$ is principal, hence $\m=(a+by)$ for some $a,b \in k[x]$.  Then $x-\lambda = (a+by)g$ and $y=(a+by)h$ for some $g,h \in R$.  Applying norms to the first equation, we have $(x-\lambda)^2 = (a^2 -abc+b^2 p)N(g)$.  Hence the right hand side has degree 2. However, by the work above $\deg(a^2 -abc+b^2p) = \max\{\deg a^2, \deg b^2 p\}$ and since $\deg p \geq 3$, it follows that $b=0$.  Thus, $y=ah$.  Say $h=u+vy$ with $u,v \in k[x]$.  Then $y=au + avy$, so that by uniqueness of representation we have $av=1$, so that $a \in k^\times$ and $\m = (a) = R$, a contradiction.
\end{proof}

\begin{rem}
We point out here some overlap between Theorem~\ref{thm:ua} and a result of Ford \cite[Proposition 3.7]{Fo-unitaffine}.  Assuming that $R$ is finitely generated over an algebraically closed field $k$, we have by Theorem~\ref{thm:geom} that $R$ is \ua\ $\iff U(R) = k^\times$.  Ford's result assumes $k=\C$ and that the polynomial $p$ is ``sufficiently general'', but his result permits $p$ to have any degree $\geq 3$, odd or even.  Then for such a polynomial $f = y^2-p(x)$, he concludes that $U(R) = k^\times$, which by our result is equivalent to unit-additivity.
\end{rem}

\begin{rem}
For the last statement of Theorem~\ref{thm:ua}, the condition that $p$ has a root in $k$ cannot be omitted.  Indeed, Brown \cite[Theorem 1.1(B)]{Br-Euaffine} provides three examples of non-Euclidean principal ideal domains of the form $k[X,Y]/(Y^2 +Y+p)$, and one of the form $k[X,Y] / (Y^2 + p)$, where $p \in k[X]$ of degree 3 or 5 and $k$ is a field of 2, 3, or 4 elements.
\end{rem}

\begin{cor}\label{cor:elliptic}
    Let $k$ be a perfect field, and let $R$ be the affine coordinate ring of an elliptic curve over $k$.  Then $R$ is a \ua\ domain whose field of units is $k$.
\end{cor}

\begin{cor}\label{cor:ellipticFTA}
Let $k$ be an algebraically closed field, let $E$ be an elliptic curve over $k$ and set $D := E \cap \A^2_k$.  Then any nonconstant polynomial mapping on $D$ admits a root.
\end{cor}

\begin{proof}
    Apply Corollary~\ref{cor:elliptic} in light of Remark~\ref{rem:FTA}.
\end{proof}

\begin{rem}\label{rem:ellbyag}
Corollary~\ref{cor:ellipticFTA} must be well-known to experts, as it follows from elementary algebraic geometry.  Indeed, let $R$ be a 1-dimensional integral domain finitely generated over $k=k^\alg$, let $D \subseteq \A^n_k$ be the corresponding affine variety, and let $E \subseteq \PP^n_k$ be its projective closure. \emph{Assume $E \setminus D$ is a singleton set $\{e\}$} (as is the case with elliptic curves as defined here). Let $f$ be a nonconstant polynomial mapping on $D$.  Since $f \notin k$, $f$ cannot be globally defined on $E$, as $\cO(E) = k$ by \cite[Theorem I.3.4]{Hart-AG}. Thus, the corresponding rational map $f \in K(E)= \Frac R$ must have a pole. However, as $f$ is defined on all of $D$, the pole must be at $e$. Similarly, $1/f \in K(E) \setminus \cO(E)$ must have a pole, which cannot be at $e$ so is at some point $d$ of $D$.  Hence, $f(d)=0$.
\end{rem}

\begin{rem}\label{rem:Ded}
It would go too far afield to show precisely which polynomials $f$ as in Lemma~\ref{lem:irreducible} give rise to Dedekind domains.  However, one can quite easily show the following, using the Jacobian criterion (see e.g. \cite[Theorem 4.4.9]{HuSw-book}):

Let $k$ be a perfect field, and let $d$ be an odd positive integer.  For any $(d+2)$-tuple $\gamma=(a,b,c_1, \ldots, c_d)$, let $f_\gamma := Y^2 + aXY + bY - (X^d + c_1 X^{d-1} + \cdots + c_{d-1}X + c_d)$.  Then there is a Zariski-dense subset $D$ of $k^{d+2}$ such that $k[X,Y]/(f_\gamma)$ is a Dedekind domain if and only if $\gamma \in D$.
\end{rem}

\begin{rem}
The case handled in \cite[Theorem 2.1]{HeiRo-utriv} is mostly subsumed by Lemma~\ref{lem:irreducible}, Theorem~\ref{thm:ua}, and Remark~\ref{rem:Ded} above.  The only case of that theorem not handled above is where $\deg p=1$.  But if $f=Y^2 + Y + p$ with $\deg p=1$, say $p(X) = -aX+b$ with $a \in k^\times$, then $k[X,Y] / (Y^2 + Y + p) \cong k[Y]$ by sending $X \mapsto (Y^2 + Y + b)/a$, which is clearly a \ua\ Euclidean domain. 
\end{rem}

In many examples of \ua\ domains $R$ from this paper, we have that there is a surjection from $R$ to the field of units.  The following is a counterexample.

\begin{example}\label{ex:HRredux}
Let $R$ and $f$ be as in Theorem~\ref{thm:ua}, with $k
=\F_2$, $c(X)=1$, and $p(X) =X^3+X^2+1$.   Then by the results of this section (or \cite[Theorem 2.1]{HeiRo-utriv}), $R$ is \ut, hence \ua\ with field of units $\F_2$.  On the other hand, by \cite[Example 2.3]{HeiRo-utriv} every residue field of $R$ contains at least four elements, so there is no map from $R$ onto $\F_2$.  The same reference shows that $R$ is a PID that is not a Euclidean domain.
\end{example}

\section{Unit dimension}\label{sec:uadim}

Let $R$ be a ring whose nilradical is a prime ideal, such that $R$ is not unit-additive. If $W$ denotes the non nilpotent sums of elements in $U(R)$, then $W$ is a multiplicatively closed subset of $R$ (see Lemma~\ref{lem:mcsums}).  It is natural to ask if the ring $W^{-1}R$ is unit-additive.  Unfortunately, this is not true in general as localization may create more units as our next example shows.   

\begin{example}
Let $R=\Z[X,2/X]$.  Then $U(R) = \{1,-1\}$.  Hence if $W$ is as above, $W = \Z \setminus \{0\}$.  In particular $W^{-1}R = \Q[X,1/X]$.  Thus $x$ is a unit of $W^{-1}R$, yet $1+X$ is not. 
\end{example}

However, we do have the following result.  Recall (see e.g. \cite[Example 4]{Mats}) that the \emph{saturation} $\tilde W$ of a multiplicative set $W \subseteq R$ consists of all divisors of every element of $W$.  That is, $\tilde W := \{x \in R \mid \exists y\in R$ such that $xy \in W\}$. We say $W$ is \emph{saturated} if $W=\tilde W$.  Recall also that $\tilde W$ is always a multiplicative set, and that $W^{-1}R \cong \tilde W^{-1}R$ canonically. 

\begin{prop}\label{pr:satsumua}
Let $R$ be a ring whose nilradical $\cN(R)$ is a prime ideal.  Let $W$ be a saturated multiplicative set that contains no nilpotent elements.  Then $T := W^{-1}R$ is \ua $\iff$ $W$ is closed under non-nilpotent sums.
\end{prop}

\begin{proof}
Suppose $W$ is closed under non-nilpotent sums.  First note that if $a/s \in T$ is a unit of $T$ (where $a\in R$ and $s \in W$), then $a \in W$ since $W$ is saturated and $a/1$ is a unit of $T$.    Suppose that  $\frac du = \frac as + \frac bw$, where $a/s$ and $b/w $ are units of $T$.  Assume $d/u$ is non-nilpotent.  Then there is some $x\in W$ such that $xswd = xuaw+xubs$.   As noted $a, b \in W$ and so $xuaw + xubs =xswd$ is a sum of elements of $W$, and hence is either nilpotent or in $W$.  But $xsw \in W$, so is non-nilpotent, and if $d$ is nilpotent then so is $d/u$, contradicting our assumption.  Since neither $xsw$ nor $d$ are in $\cN(R)$ and since the latter is a prime ideal, we have that $xswd \notin \cN(R)$.   Hence by our assumption on $W$, $xswd \in W$.      Since $W$ is saturated, $d \in W$ and so $d/u$ is a unit of $T$.

Suppose conversely that $T$ is \ua.  Let $v, w \in W$ such that $v+w$ is non-nilpotent.  If $\frac{v+w}1$ is nilpotent, then since $\cN(W^{-1}R) = W^{-1}\cN(R)$, there is some $u \in W$ with $u(v+w) \in \cN(R)$, but since $\cN(R)$ is prime and $u \notin \cN(R)$, it follows that $v+w \in \cN(R)$, contrary to assumption.  Hence, $\frac v1 + \frac w1 = \frac{v+w}1$ is non-nilpotent in $T$, so since $T$ is \ua, $\frac{v+w}1$ must be a unit.  But then there is some $\frac bu \in T$, with $b\in R$ and $u \in W$, such that $\frac{b(v+w)}u = \frac bu \frac {v+w}1 = 1_T = \frac 11$.  Thus, there is some $x \in W$ with $xb(v+w) = xu$. But since $xu \in W$, it follows that $v+w \in \tilde W = W$.
\end{proof}

\begin{lemma}\label{lem:mcsums}
Let $R$ be a ring, and let $U,V$ be multiplicatively closed subsets of $R$.  Let $W$ be the set of finite sums of elements of $V$ that lie in $U$.  That is, $W = \{v_1 + \cdots + v_n \mid n \in \N, v_j \in V \text{ for all } 1 \leq j \leq n\} \cap U$.  Then $W$ is multiplicatively closed.
In particular, the finite sums of elements of a multiplicative set that avoid a given prime ideal comprise a multiplicative set.
\end{lemma}

\begin{proof}
Let $a:=v_1 + \cdots + v_n$, $b:= t_1 + \cdots + t_k \in W$, where $v_i, t_j \in V$ for all $i,j$.  Then $ab = \sum_{i=1}^n \sum_{j=1}^k v_i t_j$.  Since $V$ is multiplicatively closed, each pairwise product $v_i t_j \in V$.  Thus, $ab$ is a finite sum of elements in $V$.  But also since $a,b \in U$, we have $ab \in U$ since $U$ is a multiplicative set.  Hence, $ab \in W$.
\end{proof}

Of course if $R$ is not \ua, then we can repeat the process of taking non nilpotent sums of units and then take the saturation of that set.   Thus we have the following definition.

\begin{defn}\label{def:udim}
Let $R$ be a ring whose nilradical is a prime ideal.  Set $W_0 = \{1\}$, and $V_0 = \widetilde{W_0} = U(R)$.  For each $i \geq 1$, we inductively define $W_i$ and $V_i$ by setting $W_i := $ the non-nilpotent sums of elements of $V_{i-1}$ (which is multiplicatively closed by Lemma~\ref{lem:mcsums}), and $V_i := \widetilde{W_i}$.
 We say the \emph{\uad} of $R$ is $n$, written $\udim R = n$, if $n$ is the smallest nonnegative integer such that $W^{-1}_nR$ is \ua.  (So, $\udim R = 0 \iff R$ is \ua.)  If there is no such $n$, we say $R$ has \emph{infinite} \uad, written $\udim R = \infty$. By Proposition~\ref{pr:satsumua}, $\udim R = \inf\{i\geq 0 \mid V_i = W_{i+1}\}.$
\end{defn}

In our next example we show that for any $n \in \N_0$, there is an integral domain $R$ with $\udim R = n$, and moreover there is a domain with $\udim R = \infty$.

\begin{example}\label{ex:uadimisdim}
We show that for each $n= 1,2...$, there exists an integral domain, finitely generated over a field (and isomorphic to a localization of a polynomial ring over a field at a single element),  whose dimension and \uad\ are both equal to $n$.  Let $k$ be a field and let      $$R= k[X_1,X_1^{-1}, X_2, (1+X_1)/X_2,...,X_n,(1+X_{n-1})/X_n].$$   Let $W_i$ and $V_i$ be as in the above definition.  It follows that $W_1$ consists of all Laurent polynomials in $X_1$ (over $k$) and that $V_1$ contains $X_2$, since $X_2(1+X_1)/X_2 = 1+X_1 \in W_1$.  Continuing in this fashion, we see that for each $i\leq n$, $W_i$ contains all polynomials in $\{X_1,X_2,...,X_i\}$, but not $X_{i+1}$, whereas $V_i$ contains $X_{i+1}$.  In fact, 
   \[
   W_i^{-1}R = k(X_1,X_2,...,X_i)[X_{i+1},X_{i+1}^{-1},..., X_n,(1+X_{n-1})/X_n].
   \]
 It follows that $\udim R =n$.

 To see that $R$ is isomorphic to the localization of a polynomial ring at a single element, let $D = k[Y_1, \ldots, Y_n]$, and consider the $k$-algebra  map $\phi: D \rightarrow R$ that sends $Y_n \mapsto X_n$ and for each $1 \leq j < n$, $Y_j \mapsto \frac{1+X_j}{X_{j+1}}$.  We define a sequence of polynomials $f_n, f_{n-1}, \ldots, f_1 \in D$ by descending induction by setting $f_n := Y_n$ and for each $1 \leq j <n$, $f_j := Y_j f_{j+1} -1$.  Note that $\phi(f_j) = X_j$ for all $j$.  Hence, $\phi(f_1)=X_1$ is a unit of $R$, so $\phi$ extends uniquely to a map $\psi: D[1/f_1] \ra R$.  On the other hand, setting $A := k[X_1, \ldots, X_n]$, let $\mu: A \ra D$ be the $k$-algebra map that sends each $X_j \mapsto f_j$.  Then $\mu(X_n)=f_n$, and for each $1\leq j < n$, we have $\mu(1+X_j) = 1+f_j = Y_j f_{j+1} = Y_j \mu(X_{j+1})$, so that $\mu$ extends naturally to $\nu: R \ra D[1/f_1]$.  We have $\nu(1/X_1) = 1/f_1$, and for each $1\leq j <n$, we have $\nu\left(\frac {1+X_j}{X_{j+1}}\right) = Y_j$.  We claim that $\psi$ and $\nu$ are inverses of each other, from which we get $R \cong D[1/f_1]$ as $k$-algebras.  For this, we need only see that $\psi$ and $\nu$ behave properly on the $X_j$s and $Y_j$s.  Accordingly, $\psi(\nu(X_j)) = \psi(f_j) = X_j$ for all $j$, $\nu(\psi(Y_n)) = \nu(X_n) = f_n=Y_n$, and for each $1\leq j <n$, we have $\nu(\psi(Y_j)) = \nu\left(\frac{1+X_j}{X_{j+1}}\right) = Y_j$.  It follows that $\dim R=n$ as well.
\end{example}

\begin{example}\label{ex:inftyua}
 On the other hand we can extend the construction of $R$ above to show that there are rings with infinite \uad, albeit ones that are not finitely generated over a field. Let $X$ be a countably infinite set of variables indexed by the natural numbers.
 Let 
   $$ R = k[X_1, X_1^{-1}, X_2,(1+X_1)/X_2,...(1+X_i)/X_{i+1},...].$$ 
Clearly $\udim R \neq n$ for any $n\in \N_0$, whence $\udim R = \infty$.   
\end{example}

In our next Theorem, we prove that unlike the ring 
in Example~\ref{ex:inftyua}, any domain that is \emph{finitely generated} over a field has finite \uad.

\begin{thm}\label{thm:uadimbound}
Let $R$ be an integral domain that is a finitely generated algebra over a field. Then $\udim R \leq \dim R$.
\end{thm}

\begin{proof}
We proceed by induction on $d=\dim R$.  When $d=0$, $R$ must itself be a field (as it is a zero-dimensional integral domain), which is \ua, hence $\udim R=0$.

Now let $d>0$ and assume we have proved the result for smaller dimensions.  Let $A$ be the subring of $R$ generated by its units.  Since $W_1$ (in the notation of Definition~\ref{def:udim}) is closed under products and nonzero sums, we have $A = W_1 \cup \{0\}$. If $A$ is itself a field, then $\udim R = 0 < d$, so we are done.  

If not, then observe that for any maximal ideal $\m$ of $R$, $A \cap \m$ is a maximal ideal of $A$.  To see this, note by Zariski's Lemma 
that the composition $k \ra A \ra R \ra R/\m$ must be a finite field extension, hence an integral extension.  This extension also factors as a pair of injective maps $k \ra A/(\m \cap A) \ra R/\m$.  Thus, $A/(\m \cap A)$ is an integral domain that is integral over a field, so by \cite[Theorem 16]{Kap-CR}, $A/(\m\cap A)$ must be a field.  Since $A$ is not a field, we have $\m \cap A \neq 0$ -- i.e., $\m \cap W_1 \neq \emptyset$.  Since every maximal ideal of $R$ intersects $W_1$, we have $\dim W_1^{-1}R \leq \dim R -1$.  Also, $L=W_1^{-1}A$ is a field, and $W_1^{-1}R$ is an integral domain finitely generated as an $L$-algebra (by the same elements that generate $R$ as a $k$-algebra). Then by inductive hypothesis, we have \[
\udim R = \udim(W_1^{-1}R) + 1 \leq \dim(W_1^{-1}R) +1 \leq \dim R. \qedhere
\]
\end{proof}

On the other hand, the difference between Krull dimension and \uad\ can be arbitrarily large, as we can glean from the following result and its corollary.

\begin{prop}\label{pr:udimsubvars}
Let $D$ be an integral domain, let $\mathbf{X} = \{X_i\}_{i \in \Lambda}$ be a set of algebraically independent indeterminates over $D$, and let $R$      be a $D$-subalgebra of $D[\mathbf{X}]$.  Then $\udim R = \udim D$.
\end{prop}

\begin{proof}
  Clearly $W_0(D) = W_0(R)= \{1\}$, and since every unit of $R$ is in $D$, we have $V_0(D) = U(D) = U(R) = V_0(R)$.   Finally observe that any divisor of an element of $D$ is already in $D$.  Hence by a straightforward induction argument, $W_i(R) = W_i(D)$ and $V_i(R) = V_i(D)$ for all $i$.  The result follows.
\end{proof}

Then from standard facts about polynomial algebras, we have the following:

\begin{cor}\label{cor:udimvars}
Let $D$ be an integral domain, let $\mathbf{X} = \{X_i\}_{i \in \Lambda}$ be a set of algebraically independent indeterminates over $D$, and set $S=D[\mathbf{X}]$.  Then $\dim S \geq \dim D + |\Lambda|$, with equality if $D$ is Noetherian.  Hence if $D$ is a finitely generated algebra over a field, we have $\dim S \geq \udim S + |\Lambda|$.
\end{cor}

\section{Rings with \uad\ 1}\label{sec:uad1}
Our next result is the setup to provide some broad classes of rings with \uad\ 1.

\begin{prop}\label{pr:Jacudim}
    Let $R$ be a ring whose nilradical $\cN(R)$ is prime.  Suppose $\Jac(R) \neq \cN(R)$.  Then $\udim R=1$.
\end{prop}

\begin{proof}
By Example~\ref{ex:Jacnil}, $\udim R >0$.  On the other hand, let $W$ be the multiplicative set of sums of units that are not nilpotent (see Lemma~\ref{lem:mcsums}).  Take any $j \in \Jac(R) \setminus \cN(R)$ and $f\in R\setminus \cN(R)$; then $1+jf$ is a unit, so $jf= (-1)+(1+jf) \in W$ since $j \notin \cN(R)$ and $f \notin \cN(R)$.  Thus, $f$ is in the saturation of $W$, so $\frac f1$ is a unit in $W^{-1}R$.  It follows that $W^{-1}R = R_{\cN(R)}$ is a zero-dimensional local ring, so $\udim R \leq 1$. Thus, $\udim R=1$.
\end{proof}

\begin{cor}\label{cor:Jacudim}
Any ring $R$ of the following types has $\udim R=1$, provided that it has prime nilradical (e.g. if it is a domain). \begin{enumerate}
    \item\label{it:semiloc} $R$ is a semilocal ring such that all maximal ideals have positive height (e.g. a local ring of positive Krull dimension).
    \item\label{it:Gdom} $R$ is a G-domain that is not a field.
    \item\label{it:ps} $R = A[\![X]\!]$, where $A$ is a nonzero ring and $X$ an analytic indeterminate over $A$.
    \item\label{it:locpull} $R$ is a subring of $T$ and contains $\m$, where $(T,\m)$ is a local ring of positive Krull dimension.
\end{enumerate}
\end{cor}

\begin{proof}
By Proposition~\ref{pr:Jacudim}, it is enough to show that each of the above has $\Jac(R) \neq \cN(R)$.

Case (\ref{it:semiloc}): 
Let $\m_1, \ldots, \m_s$ be the maximal ideals of $R$.  For each $1\leq i \leq s$, let $x_i \in \m_i \setminus \cN(R)$.  Set $t := \prod_{i=1}^s x_i$.  Then $t \in \Jac(R) \setminus \cN(R)$.

Case (\ref{it:Gdom}):  Recall from \cite[Section 1-3]{Kap-CR} that this means $R$ is an integral domain that admits a nonzero element $u$ that is in every nonzero prime ideal.  In particular, $u$ is in every maximal ideal, but avoids the nilradical of the ring (which is $(0)$), and is thus the desired element of $\Jac(R)$.

Case (\ref{it:ps}): We have $X \in \Jac(R) \setminus \cN(R)$, since for any $g\in R$, $1+Xg \in U(R)$.

Case (\ref{it:locpull}): Let $\pi: T \onto T/\m =:k$ be the canonical map, and let $D = \pi(R)$.  Note that $R =\pi^{-1}(D)$, since $\pi$ is surjective and its kernel is inside $R$.  But as $\m$ is an ideal of $T$ contained in $R$, we then also see that $\m$ is an ideal of $R$.

Now let $t \in \m$ and $f\in R$.  Then as $\m = \Jac(T)$, $1+tf$ is a unit of $T$. Say $(1+tf)g =1$, $g\in T$.  Then $g = 1-tfg \in 1+\m \subseteq R$, so that $1+tf$ is a unit of $R$.  Hence, $t \in \Jac(R)$.  Thus, $\m \subseteq \Jac(R)$.

Finally, since $\dim T>0$, $\m$ contains non-nilpotent elements.  Thus, $\Jac(R) \supseteq \m \supsetneq \cN(R)$.
\end{proof}

Next we see, in a family of examples of interest in algebraic number theory, that not every domain with \uad\ 1 satisfies the conditions of Proposition~\ref{pr:Jacudim}. 

\begin{example}\label{ex:algint}
Let $A$ be an integral domain that is generated by its units (i.e., every element is a sum of units).\footnote{Note that $\Z$ is such a ring, as are the rings of integers of $\Q(i), \Q(\sqrt{-3}),$ and $\Q(\sqrt{3})$, but not of $\Q(\sqrt{-5})$ or $\Q(\sqrt{5})$; see \cite[Theorems 7 and 8]{AsVam-usum}.}   Let $K$ be its fraction field, $L$ an algebraic closure of $K$, and $R$ a ring between $A$ and $L$.  Then $\udim R\leq 1$.  To see this, let $0\neq r \in R$.  Then $r$ is algebraic over $K$.  By \cite[Theorem 15]{Kap-CR},
$r^{-1} \in K[r]$.  But $W_1$ (from Definition~\ref{def:udim}, with respect to $R$) contains all the nonzero elements of $A$.  Hence, $K \subseteq W_1^{-1}R,$
 whence $r^{-1} \in K[r] \subseteq W_1^{-1}R$.  Thus, $W_1^{-1}R$ is a field.

    In particular, let $R$ be a subring of the ring of all algebraic integers.  Then $R$ is not \ua, since $2$ is a nonzero nonunit sum of units in $R$. Thus $\udim R \geq 1$.  But $R$ is an integral extension of $\Z$, so by the above we have $\udim R \leq 1$ (so $\udim R=1$).  Then since $\Z$ is a Jacobson ring, its integral extension $R$ must also be a Jacobson ring.  Therefore, $R$ cannot satisfy the conditions of Proposition~\ref{pr:Jacudim}. 
\end{example}

\section{Unit-additive closure and localizations}\label{sec:uac}

Let $R$ be a ring such that $\cN(R)$ is a prime ideal. We show that there exists a localization of $R$  which has a universal mapping property with respect to maps from $R$ to unit-additive rings.  In case $R$ is an integral domain this ring has an alternative construction as the unique minimal unit-additive overring of $R$.

\begin{thm}\label{thm:ualoc}
Let $R$ be a ring whose nilradical $\cN(R)$ is a prime ideal.  Then there is a multiplicative set $W \subseteq R \setminus \cN(R)$ such that $W^{-1}R$ is \ua, and such that for any ring homomorphism $\phi: R \ra S$, where $S$ is \ua\ and $\phi^{-1}(\cN(S)) \subseteq \cN(R)$, there is a unique ring homomorphism $\tilde \phi: W^{-1}R \ra S$ such that $\tilde \phi \circ \ell = \phi$, where $\ell: R \ra W^{-1}R$ is the localization map.
\end{thm}

\begin{proof}
Let $W_i$, $V_i$, for $i\in \N_0$ be as in Definition~\ref{def:udim}.  Set $W := \bigcup_{n \in \N_0} W_n$.  Since $W_0 \subseteq V_0 \subseteq W_1 \subseteq V_1 \subseteq \cdots$, we also have $W = \bigcup_{n \in \N_0} V_n$.

\vspace{.1in}
\noindent \textbf{Claim:} $W$ is a saturated multiplicative set that is closed under non-nilpotent sums.

\begin{proof}[Proof of claim]
To see that $W$ is multiplicative, let $v,w \in W$.  Then there exist $i,j$ such that $v\in W_i$, $w \in W_j$.  Let $n= \max\{i,j\}$.  Then $v,w \in W_n$, so $vw \in W_n$ (since $W_n$ is multiplicative), but $W_n \subseteq W$, so $vw \in W$.

To see that $W$ is saturated, let $a,b \in R$ such that $ab \in W$.  Then there is some $n$ such that $ab \in V_n$.  Since $V_n$ is saturated, we have $a,b \in V_n \subseteq W$.

To see that $W$ is closed under non-nilpotent sums, let $w_1, \ldots, w_n \in W$ such that $w := w_1 + \cdots + w_n$ is non-nilpotent.  For each $1\leq i \leq n$, there exists $k_i \in \N_0$ such that $w_i \in W_{k_i}$.  Let $k=\max\{k_i \mid 1\leq i \leq n\}$.  Then $w_1, \ldots, w_n \in W_k$.  But $W_k$ is closed under non-nilpotent sums, so that since $w$ is non-nilpotent, $w \in W_k \subseteq W$, finishing the proof of the Claim.
\end{proof}

Then by Proposition~\ref{pr:satsumua}, $W^{-1}R$ is \ua.

Let $\phi: R \ra S$ be a ring homomorphism and suppose $S$ is \ua\ and $\phi^{-1}(\cN(S)) \subseteq \cN(R)$. Let $w\in W$.  Then there is some $n$ such that $w\in W_n$.  We want to show by induction on $n$ that $\phi(w)$ is a unit of $S$.

To see this, first take the case $n=0$.  Then $w=w_1 + \cdots + w_k$, where each $w_i$ is a unit of $R$.  But units map to units in ring homomorphisms, so $\phi(w) = \phi(w_1) + \cdots + \phi(w_k)$ is a unit or nilpotent.  But since $\phi^{-1}(\cN(S)) \subseteq \cN(R)$ and $w$ is not nilpotent, it follows that $\phi(w)$ is not nilpotent. Hence it is a unit.

Now let $n>0$ and assume that for any $x\in W_{n-1}$, $\phi(x)$ is a unit.  We have $w = v_1 + \cdots + v_k$, where each $v_i \in V_{n-1}$.  But since $V_{n-1}$ is the saturation of $W_{n-1}$, there exist $c_i \in R$ such that $c_i v_i \in W_{n-1}$.  Then by induction $\phi(c_i v_i) = \phi(c_i)\phi(v_i)$ is a unit, whence $\phi(v_i)$ is a unit.  Also, since $w$ is non-nilpotent, $\phi(w)$ is non-nilpotent, so $\phi(w) = \sum_{i=1}^k \phi(v_i)$ is a non-nilpotent sum of units in $S$, hence is a unit.  This completes the induction and shows that $\phi$ maps all elements of $W$ to units of $S$.  The rest follows from standard localization theory.
\end{proof}

Note that if $R$ is an integral domain, the condition ``$\phi^{-1}(\cN(S)) \subseteq \cN(R)$'' reduces to the condition that the ring homomorphism is injective.

\begin{defn}
Let $S$ be a ring and $R$ a subring.  The \emph{\ua\ closure} of $R$ in $S$, denoted $R^{\uasym}_S$, is the intersection of all the \ua\ subrings of $S$ that contain $R$.
\end{defn}

\begin{lemma}\label{lem:isec}
Let $S$ be a commutative ring.  The intersection of any nonempty collection of \ua\ subrings of $S$ is \ua.
\end{lemma}

\begin{proof}
Let $\{R_\alpha\}_{\alpha \in \Lambda}$ be such a collection, and set $R := \bigcap_\alpha R_\alpha$.  Let $u,v \in U(R)$, and suppose $u+v$ is not nilpotent.  Then $u^{-1}, v^{-1} \in R$, so for all $\alpha \in \Lambda$, we have $u,v,u^{-1}, v^{-1} \in R_\alpha$.  So for each $\alpha$, $u,v$ are units of $R_\alpha$, so $u+v$, as it is not nilpotent, must be a unit of $R_\alpha$.  That is, $(u+v)^{-1} \in R_\alpha$ for all $\alpha$, whence $(u+v)^{-1} \in \bigcap_\alpha R_\alpha = R$, so $u+v$ is a unit of $R$.
\end{proof}

A quick thing to conclude is that if $L/K$ is any field extension and $R$ is a subring of $K$, then $R^{\uasym}_K = R^\uasym_L$, since $K$ is also a \ua\ subring of $L$.  Hence we may unambiguously write $R^\uasym$ when our focus is on the ring rather than the field containing it.

More significantly, the above lemma shows that when $S$ is itself \ua,  the operation $(-)^{\uasym}_S$ really is a \emph{closure operation} on the poset of subrings of $S$, ordered by inclusion (see \cite[Definition 0.26(7)]{El-clbook}), and $R^{\uasym}_S$ is always \ua\ for any subring $R$ of $S$.  But this operation is hard to compute from the definition, and in fact for integral domains, we have the tools to compute it above.

\begin{prop}\label{pr:uacl}
Let $R$ be an integral domain, and let $W \subseteq R$ be as constructed in Theorem~\ref{thm:ualoc}.  Then $R^\uasym = W^{-1}R$.
\end{prop}

\begin{proof}
By Theorem~\ref{thm:ualoc}, $W^{-1}R$ is a \ua\ subring of $K=\Frac R$ that contains $R$.  Since $R^\uasym$ is the intersection of all such subrings of $K$, we have $W^{-1}R \supseteq R^\uasym$.

On the other hand, since $R^\uasym$ is \ua\ by Lemma~\ref{lem:isec}, again by Theorem~\ref{thm:ualoc} there is a unique ring homomorphism $\tilde i: W^{-1}R \ra R^\uasym$ that extends the inclusion map $i: R \ra R^\uasym$.  For any $w\in W$, we have $1=\tilde i(1) = \tilde i(w/w) = \tilde i(w \cdot \frac 1w) = \tilde i(w) \tilde i(1/w) = w \tilde i(1/w)$. Thus, $\tilde i(1/w) = w^{-1} \in R^\uasym$.  Thus, $R^\uasym \supseteq R[W^{-1}] = W^{-1}R$.
\end{proof}

We end with a connection of unit-additive closure to the inequality in Theorem~\ref{thm:uadimbound}.
\begin{prop}
Let $R$ be a domain that is finitely generated over a field.  If $\udim R = \dim R$, then $R^\uasym = $ the fraction field of $R$.
\end{prop}

\begin{proof}
Let $d = \dim R = \udim R$.  By the proof of Theorem~\ref{thm:uadimbound}, for each $1 \leq i \leq d$ we have $\dim W_i^{-1}R < \dim W_{i-1}^{-1}R$, so that since Krull dimension is an integer, we have $\dim W_i^{-1}R \leq \dim W_{i-1}^{-1}R -1$.  Thus, $0 \leq \dim R^\uasym = \dim W_d^{-1}R \leq \dim W_0^{-1}R - d = \dim R - d = 0$.  That is, $R^\uasym$ is a zero-dimensional integral domain, thus a field.  But as it is also a localization of $R$, it follows that it must be the fraction field of $R$.
\end{proof}

\section{Questions}\label{sec:questions}
A new subject deserves some entry points.  Hence, we present here five questions for further exploration.

\begin{question}
Let $k$ be a field and $R$ a finitely generated $k$-algebra that is a \ua\ Euclidean domain.  Does it follow that $R \cong k$ or $k[X]$?
\end{question}

By Theorem~\ref{thm:Eucl}, we know this both when $k$ is algebraically closed and when $k=\F_2$ is a maximal subfield of $R$.

\begin{question}
    Characterize when the monoid algebra $A[M]$ is \ua, when $M$ is not necessarily cancellative.
\end{question}

All the cancellative cases were handled in Theorem~\ref{thm:monalg}.

\begin{question}
    Let $R$ be a ring with prime nilradical.  In what generality is it true that $\udim R \leq \dim R$?
\end{question}

This inequality holds when $R$ is an integral domain finitely generated over a field (see Theorem~\ref{thm:uadimbound}).  It also holds for rings of \uad\ $\leq 1$ as a zero-dimensional ring with prime nilradical is always \ua.

\begin{question}\label{q:geoudim}
    Find a geometric interpretation of \uad\ for domains that are finitely generated over a field.
\end{question}

Theorem~\ref{thm:puncaone} is such an interpretation for \uad\ \emph{zero}.  But for any higher integer, we do not know a geometric interpretation.

\begin{question}
    Let $R \subseteq S$ be an integral (or module-finite) extension of integral domains, let $n\in \N_0$, and suppose $\udim S \leq n$.  Does it follow that $\udim R \leq n$?
\end{question}

Proposition~\ref{pr:intext} answers the above question positively when $n=0$.

\section*{Acknowledgment}
We would like to take this opportunity to thank several people who made contributions to the paper, listed below: Karl Schwede showed us why the argument proving Theorem~\ref{thm:Eucl}(2) doesn't generalize to all infinite fields (see Remark~\ref{rem:Schwede}). He also pointed out Remark~\ref{rem:ellbyag}. Takumi Murayama pointed us to Chevalley's theorem, which allowed us to complete the proof of Theorem~\ref{thm:geom}. 
 Anton Lukyanenko asked a question that led to us finding Example~\ref{ex:HRredux}.  Swan Klein asked Question~\ref{q:geoudim} and corrected an indexing problem in Definition~\ref{def:udim}.

We also offer a sincere thanks to the anonymous referee, who read over our paper carefully, caught several minor errors, pointed out references in the literature of which we were not aware, and made suggestions that improved the exposition.

\providecommand{\bysame}{\leavevmode\hbox to3em{\hrulefill}\thinspace}
\providecommand{\MR}{\relax\ifhmode\unskip\space\fi MR }
\providecommand{\MRhref}[2]{%
  \href{http://www.ams.org/mathscinet-getitem?mr=#1}{#2}
}
\providecommand{\href}[2]{#2}

\end{document}